\documentclass[12pt]{amsart}

\usepackage{enumitem}
\usepackage{color}
\usepackage[normalem]{ulem} 

\newcommand{\abs}[1]{\left\lvert#1\right\rvert}
\newcommand{\CCC}{\mathcal{C}}
\newcommand{\BBB}{\mathcal{B}}
\newcommand{\LLL}{\mathcal{L}}
\newcommand{\DDD}{\mathcal{D}}
\newcommand{\GGG}{\mathcal{G}}
\newcommand{\MMM}{\mathcal{M}}
\newcommand{\FFF}{\mathcal{F}}
\newcommand{\EEE}{\mathcal{E}}

\newcommand{\RR}{\mathbb{R}}
\newcommand{\NN}{\mathbb{N}}
\newcommand{\ZZ}{\mathbb{Z}}
\newcommand{\zz}{\mathbf{z}}
\newcommand{\hd}{\hat{d}}
\newcommand{\ph}{\varphi}
\newcommand{\Ms}{\MMM_\sigma}
\newcommand{\Mse}{\MMM_\sigma^e}
\newcommand{\ww}{\mathbf{w}}

\newcommand{\smatrix}[1]{\left(\begin{smallmatrix} #1 \end{smallmatrix}\right)}
\newcommand{\ulim}{\varlimsup}
\newcommand{\Phieg}{\Phi}

\theoremstyle{plain}

\newtheorem{thma}{Theorem}
\newtheorem{theorem}{Theorem}[section]

\newtheorem{proposition}[theorem]{Proposition}
\newtheorem{lemma}[theorem]{Lemma}
\DeclareMathOperator{\Ham}{Ham}

\theoremstyle{definition}
\newtheorem{definition}[theorem]{Definition}
\newtheorem{remark}[theorem]{Remark}

\numberwithin{equation}{section}

\begin{document}

\title[Large deviations and non-uniform structures]{Large deviations for systems with non-uniform structure}

\author[V. Climenhaga]{Vaughn Climenhaga}
\address{Department of Mathematics \\ University of Houston \\
Houston, Texas 77204, USA}
\email{climenha@math.uh.edu}

\author[D. J. Thompson]{Daniel J. Thompson}
\address{Department of Mathematics \\ The Ohio State University \\
100 Math Tower, 231 West 18th Avenue, Columbus, Ohio 43210,  USA}
\email{thompson@math.osu.edu}

\author[K. Yamamoto]{Kenichiro Yamamoto}
\address{Department of General Education, Nagaoka University of Technology,
Niigata 940-2188, Japan
}
\email{
k\_yamamoto@vos.nagaokaut.ac.jp}
\thanks{V.C.\ was partially supported by NSF grant DMS-1362838.  D.T.\ was partially supported by NSF grants DMS-$1101576$ and DMS-$1259311$}

\begin{abstract}
We use a weak Gibbs property and a weak form of specification to derive level-2 large deviations principles for symbolic systems equipped with a large class of reference measures.  This has applications to a broad class of symbolic systems, including $\beta$-shifts, $S$-gap shifts, and their factors.  A crucial step in our approach is to prove a `horseshoe theorem' for these systems. 
\end{abstract}

\date{\today}

\maketitle

\section{Introduction}

We introduce criteria for a symbolic system to satisfy the large deviations principle. These criteria are motivated by the `non-uniform' structure of our main examples -- $\beta$-shifts, $S$-gap shifts, and their factors -- but apply more generally. We prove the following main result.  (See \S\ref{sec:defs} for precise definitions.)

\begin{thma}\label{thm:main}
Let $(X,\sigma)$ be a shift on a finite alphabet, $m$ a Borel probability measure on $X$, and $\ph\colon X\to\RR$ a continuous function.  Let $\LLL$ be the language of $X$.  Suppose there exists a set $\GGG\subset\LLL$ such that
\begin{enumerate}[label=\textbf{[A.\arabic{*}]}]
\item\label{A.spec}there exists $\tau >0$ such that for every $v,w\in \GGG$ there exists $u\in \LLL$ with $|u|\leq \tau$ such that $vuw\in \GGG$; 
\item\label{A.ASD} $\LLL$ is edit approachable by $\GGG$;
\item\label{A.Gibbs} $m$ is Gibbs for $\ph$ with respect to the collection $\GGG$.
\end{enumerate}
Then $(X,\sigma)$ satisfies a level-2 large deviations principle with reference measure $m$ and rate function $q^\ph\colon\MMM(X)\to[-\infty,0]$ given by
\begin{equation}\label{eqn:rate}
q^\ph(\mu) = \begin{cases}
h(\mu) + \int\ph\,d\mu - P(\ph) &\mu\in\Ms(X),\\
-\infty&\text{otherwise}.
\end{cases}
\end{equation}
\end{thma}
Condition \ref{A.spec} is a form of the specification property for $\GGG$.  Condition \ref{A.ASD} means that any word $w\in\LLL$ can be transformed into a word in $\GGG$ without making too many edits.  The condition \ref{A.Gibbs} means that $m$ satisfies an upper Gibbs bound on all cylinders, and a lower Gibbs bound on cylinders corresponding to words in $\GGG$.  Level-2 large deviations gives an exponential decay rate for the measure of the set of points whose empirical averages are experiencing a `large deviation' from their expected value (see \S \ref{intro:ldp}).  

The criteria we introduce can be verified for many shift spaces including $S$-gap shifts, $\beta$-shifts, and their factors, using a large class of equilibrium states as reference measures. For $\beta$-shifts, large deviations was previously known only in the case that the reference measure is the measure of maximal  entropy (MME) \cite{PfS}.  For $S$-gap shifts, and subshift factors of $\beta$-shifts and $S$-gap shifts, these are the first results on large deviations. 

The recent work of  \cite{CT, CT2} has provided the necessary groundwork to make our criteria verifiable for a large class of symbolic systems, by giving conditions for a potential $\ph$ to have a unique equilibrium state with the weak Gibbs property \ref{A.Gibbs} -- we state these in Theorem \ref{thm:CT}. In particular, for $\beta$-shifts, we proved in \cite{CT2} that we have a unique equilibrium state with the weak Gibbs property for every H\"older continuous potential. In this paper, we show the same result for $S$-gap shifts.

Our approach belongs to the `orbit-gluing approach' to large deviations, which relies on direct constructions based on the specification property (or one of its variants).  Classic and recent references include \cite{FO, EKW, Y, PfS, Ya}. The basic strategy is
\begin{itemize}
\item to obtain a (weak) Gibbs bound for the reference measure using constructive techniques;
\item to establish the entropy density of ergodic measures.
\end{itemize}
Our approach is related to powerful general techniques of Pfister and Sullivan \cite{PfS}, who  introduced two hypotheses from which large deviations follow: the approximate product property for the system, and the existence of (upper and lower) weak energy functions for the reference measure. Their results show that a uniform upper Gibbs property yields the upper large deviations bound, and we apply that result in this paper  (see \S \ref{sec:ub}). However, their hypotheses for the lower large deviations bound are harder to verify;  Pfister and Sullivan used an ad hoc argument in the case of the MME of a $\beta$-transformation, but this approach has so far not been extended to other settings. 
Our approach to the lower large deviations bound is similar in philosophy to Pfister and Sullivan but has some important novelties:
\begin{enumerate}
\item An axiomatic approach which explicitly shows how to derive large deviations results from a weak Gibbs property which is known to hold for a large class of equilibrium states;
\item Hypotheses which in many examples are more convenient to verify than those of \cite{PfS};
\item A crucial step in our proof is to establish entropy density of the ergodic measures \emph{supported on horseshoes} (see Theorem \ref{prop:ent-appr0}). This is a point of independent interest in our approach, which is analogous to the Katok horseshoe theorem in smooth dynamics.
\end{enumerate}
We discuss the relationship between our hypotheses and those of Pfister and Sullivan in \S \ref{sec:lowerenergy}.
We expect  that our approach can be adapted to non-symbolic topological dynamical systems, where analogues of some of the ideas in this paper have been introduced in \cite{CT3, CT4}. 

\subsection*{A `horseshoe' theorem}

The following result is a key step in our proof of Theorem~\ref{thm:main}, and clarifies item (3) from the previous section. A more precise version of this statement is proved as Proposition \ref{prop:ent-appr}.

\begin{thma}\label{prop:ent-appr0}
Let $X$ be a shift space and suppose that $\GGG\subset \LLL(X)$ satisfies~\ref{A.spec} and~\ref{A.ASD}.
Then there exists a family $\{X_n\}_{n\in \NN}$ of transitive sofic subshifts of $X$ such that every invariant measure on $X$ is entropy approachable by ergodic measures on $X_n$; that is, for any $\eta>0$, any $\mu\in \Ms(X)$, and any neighborhood $U$ of $\mu$ in $\Ms(X)$, there exist $n\ge 1$ and $\mu'\in\Mse(X_n)\cap U$ such that $h(\mu')>h(\mu)-\eta$.  
\end{thma}

\subsection*{Comparison with other approaches}
Large deviations results in dynamics have received a great deal of attention since their introduction in the 1980's \cite{OP1, OP,T1,T2, De, Ki1, Y, AP, E, E2}.  We recommend the introductions of \cite{Va, MN, RBY,HR, ChT1} for extensive references and discussion. In particular, the `functional approach' and the `tower approach' are powerful alternatives to the `orbit-gluing' approach developed here.  

The `functional approach' as described in \cite{HR} relies on differentiability of a certain functional on the space of observables and this approach has been used successfully in e.g. \cite{Ki1, T1, T2, Lo, HR} for examples including rational maps of the Riemann sphere. 
The key requirements are a weak version of the Gibbs property, such as \cite[(1.5)]{HR}, and the existence of a dense subspace $W\subset C(X)$ such that every $\psi\in W$ has a unique equilibrium state. 
For $\beta$-shifts and $S$-gap shifts, the dense subspace condition is satisfied, but we do not see any obvious way to verify \cite[(1.5)]{HR} for these examples.

The `tower approach' relies on relating the original system to a countable state Markov shift via a tower construction, and yields results on both exponential and sub-exponential rates of decay. These results typically use an SRB reference measure. The case of more general reference measures is largely unexplored using this approach. Classic and recent references include \cite{RBY, KN, PSY, PS, MN, Me, Ch, ChT1}.


\subsection*{Layout of the paper} 
In \S \ref{sec:defs}, we establish our definitions.  In \S \ref{sec:consequences}, we give various consequences of Theorem \ref{thm:main} using the thermodynamic results developed in \cite{CT,CT2}, including applications to $\beta$-shifts, $S$-gap shifts, and their factors. In \S \ref{sec:pfs}, we prove Theorem \ref{thm:main}. In \S \ref{sec:apps}, we give proofs that the examples ($\beta$-shifts and $S$-gap shifts) satisfy the conditions of Theorem \ref{thm:main}. In \S \ref{sec:lowerenergy}, we make the connection between our hypotheses and the weak lower energy functions of Pfister and Sullivan. The proofs of lemmas which are not proved in the body of the text appear in \S \ref{sec:lemmas}. 

\subsection*{Acknowledgments}

We are grateful to Zheng Yin for pointing out an error in an early version of this paper. 
We would also like to thank Henri Comman for useful discussions, and the anonymous referee for helpful suggestions.

\section{Definitions}\label{sec:defs}
Let $A$ be a finite set and $A^\NN$ (resp.\ $A^\ZZ$) be the set of all one-sided (resp.\ two-sided) infinite sequences on the alphabet $A$, with the standard metric $d(x,y) = 2^{-t(x,y)}$, where $t(x,y) = \min \{ |k|\mid x_k\neq y_k\}$. 
The shift map on $A^\NN$ is $\sigma\colon x_1x_2\cdots \mapsto x_2x_3\cdots$, and the shift map on $A^\ZZ$ is defined analogously.  A subshift is a closed $\sigma$-invariant set $X\subset A^\NN$ or $X\subset A^\ZZ$. All of the results and proofs in this paper apply equally to one-sided and two-sided shifts, so we treat both cases simultaneously.

\subsection{Large deviations principles} \label{intro:ldp}
The large deviations principle referred to in Theorem \ref{thm:main} describes the rate of convergence of empirical averages relative to a fixed reference measure $m$. More precisely, one writes $\delta_y$ for the Dirac measure concentrated at $y$, and $\mathcal{E}_n(x) := \frac 1n \sum_{j=0}^{n-1} \delta_{\sigma^jx}$ for the \emph{empirical measure} associated to the orbit segment $x,\sigma(x),\dots,\sigma^{n-1}(x)$.  If $m$ is $\sigma$-invariant and ergodic, then $\mathcal{E}_n(x) \to m$ for $m$-a.e.\ $x$, and one can quantify this convergence by studying the rate of decay of $m\{x\mid \mathcal{E}_n(x)\in U\}$, where $U$ is a suitable subset of the space of all probability measures on $X$. For the systems studied here, 
this quantity decays exponentially in $n$ whenever $m\notin U$; thus the goal is to describe \emph{rate functions} $\underline r(U)$ and $\overline r(U)$ that bound the lower and upper limits of $\frac 1n \log m\{x \mid \mathcal{E}_n(x) \in U\}$.
This type of result is called level-2 large deviations. One may also consider level-1 large deviations and study $m\{x\mid \frac 1n S_n\ph(x)\in V\}$ for some fixed observable $\ph$ and some $V\subset \mathbb{R}$.  For continuous $\ph$, level-2 results  imply level-1 results via the contraction principle (see \cite{HR}). 

We now state the level-2 large deviations principle precisely. 
Denote by $\mathcal{M}(X)$ the set of all Borel probability measures on $X$ with the weak* topology.  This topology is induced by the metric
\[
D(\mu,\nu):=\sum_{n=1}^{\infty}\frac{|\int\varphi_n\, d\mu-\int\varphi_n \,d\nu|}{2^{n+1}\|\varphi_n\|_{\infty}},
\]
where $\{\ph_n\}\subset C(X)$ is a countable dense subset.  Let $\MMM_\sigma(X)\subset \MMM(X)$ be the set of $\sigma$-invariant measures, and let $\MMM_\sigma^e(X)\subset \MMM_\sigma(X)$ be the set of ergodic measures.

\begin{definition}\label{def:LDP}
We say that the system $(X,\sigma)$ satisfies a \emph{level-2 large deviations principle} with a reference measure $m\in\mathcal{M}(X)$ and a rate function $q\colon\mathcal{M}(X)\to [-\infty,0]$ if $q$ is upper semicontinuous,
\begin{displaymath}
\varliminf_{n\rightarrow\infty}\frac{1}{n}\log m\left(\left \{x\in X
\mid \EEE_n(x) \in U \right \}\right)\ge\sup_{\mu\in U}q(\mu)
\end{displaymath}
holds for any open set $U\subset\mathcal{M}(X)$, and
\begin{displaymath}
\varlimsup_{n\rightarrow\infty}\frac{1}{n}\log m\left(\left \{x\in X
\mid \EEE_n(x) \in F \right \}\right)\le\sup_{\mu\in F}q(\mu)
\end{displaymath}
holds for any closed set $F\subset\mathcal{M}(X)$. 
\end{definition}

\subsection{Languages and decompositions}

The \emph{language of $X$}, denoted by $\LLL = \LLL(X)$, is the set of finite words that appear in some $x\in X$ -- that is,
\[
\LLL(X) = \{w\in A^* \mid [w]\neq \emptyset\},
\]
where $A^* = \bigcup_{n\geq 0} A^n$ and $[w]$ is the central cylinder for $w$, which in the one-sided case is the set of sequences $x\in X$ that begin with the word $w$. Given $w\in\LLL$, let $|w|$ denote the length of $w$. For any collection $\DDD \subset \LLL$, let $\DDD_n$ denote $\{w\in \DDD \mid |w| = n\}$. Thus, $\LLL_n$ is the set of all words of length $n$ that appear in sequences belonging to  $X$. Given words $u, v$, we use juxtaposition $uv$ to denote the word obtained by concatenation. 

A \emph{decomposition} for $\LLL$ is a choice of three sets of words $\CCC^p,\GGG,\CCC^s\subset \LLL$, together with a map from $\LLL$ to $\CCC^p \times \GGG \times \CCC^s$ which assigns to a word $w\in\LLL$ a triple of words $u^p\in\CCC^p, v\in\GGG,u^s\in\CCC^s$ such that $w=u^pvu^s$.
We write $\LLL = \CCC^p \GGG \CCC^s$ when the language can be decomposed in this way.\footnote{In \cite{CT}, we defined a decomposition slightly differently - there, we assumed there exists (possibly multiple) ways to decompose each word. Our definition there did not carry the information on \emph{which} decomposition to use for a given word like we do here. See Kwietniak, Oprocha and Rams \cite{KOR} for a clarification of this issue.}  We make a standing assumption that $\emptyset \in \CCC^p, \GGG, \CCC^s$ to allow for words in $\LLL$ that belong purely to one of the three collections (this is also implicit in \cite{CT, CT2}).

Once a decomposition $\LLL= \CCC^p\GGG\CCC^s$ has been fixed, we consider for each $M\in\NN$ the set
\begin{equation}\label{eqn:GM}
\GGG^M := \{ w \mid \mbox{the decomposition } w=u^pvu^s \mbox{ has } |u^p|, |u^s| \leq M\}.
\end{equation}
Note that $\LLL = \bigcup_{M\in\NN} \GGG^M$, so this defines a filtration of the language.

\subsection{Entropy and pressure for shift spaces}

Given a collection $\DDD\subset \LLL$, the \emph{entropy} of $\DDD$ is
\[
h(\DDD) := \ulim_{n\to\infty} \frac 1n \log \#\DDD_n,
\]
where $\DDD_n = \{w\in\DDD \mid |w|=n\}$.  The entropy of an invariant measure $\mu\in\Ms(X)$ is $
h(\mu) := \lim_{n\to\infty} \frac 1n \sum_{w\in\LLL_n} -\mu[w]\log\mu[w].
$
For a fixed potential function $\ph\in C(X)$, the \emph{pressure} of $\DDD\subset \LLL$ is 
\[
P(\DDD,\ph) := \ulim_{n\to\infty} \frac 1n \log \Lambda_n(\DDD,\ph),
\]
where
\[
\Lambda_n(\DDD,\ph) = \sum_{w\in\DDD_n} e^{\sup_{x\in[w]} S_n\ph(x)}
\]
and $S_n\ph(x) = \sum_{k=0}^{n-1} \ph(\sigma^kx)$.  We write $P(\ph) = P(\LLL,\ph)$.

We will be primarily concerned with potentials having some extra regularity: we say that $\ph$ has the \emph{Bowen property on $\DDD$} if there is $V\in\RR$ such that for every $n\in\NN$, every $w\in\DDD_n$, and every $x,y\in [w]$, we have $|S_n\ph(x) - S_n\ph(y)| \leq V$.  In particular, if $\ph$ is H\"older continuous then it has the Bowen property on every $\DDD\subset \LLL$.

\subsection{Specification}

We define the specification properties that appear in this paper, and the relationships between them.
\begin{definition}\label{def:spec}
Given a shift space $X$ and its language $\LLL$, consider a subset $\GGG \subset \LLL$.  Given $\tau\in \NN$, we say that $\GGG$ has \emph{(W)-specification with gap length $\tau$} if for every $v,w\in \GGG$ there is $u\in \LLL$ such that $vuw\in \GGG$ and $|u|\leq \tau$.
\end{definition} 
In the case $\GGG = \LLL$, this is equivalent to the well-known weak specification property for the shift. If the gluing word $u$ can always be taken to have length exactly $\tau$, we say that $\GGG$ has (S)-specification.  An important special case, which corresponds to specification with gap length $0$, is the following. 
\begin{definition}\label{def:morespec}
We say that $\GGG\subset \LLL$ has 
the \emph{free concatenation} property 
if for all $u, w \in \GGG$, we have $uw\in \GGG$.
\end{definition}

\begin{remark}
The definition of specification on $\GGG$, rather than on all of $\LLL$, was introduced by the first two authors in \cite{CT}.  The definition above has an important difference from the (W)-specification property in \cite{CT}.  There, we made the weaker requirement that a finite collection of words $w^1, \dots w^n \in \GGG$ can be glued to form a word in $\LLL$; that is, there exist words $u^1, \ldots, u^{m-1} $ with length at most $\tau$ such that $w^1u^1w^2u^2 \cdots u^{m-1}u^m \in \LLL$. 
We require the stronger property of Definition \ref{def:spec} for the arguments in \S\ref{free:concat}, which allow us to replace $\GGG$ with a collection $\FFF$ having the free concatenation property; this is necessary for our construction of horseshoes. The stronger property is natural in the symbolic setting, and is satisfied for all our motivating examples.
\end{remark}

\subsection{Edit Approachability}
First we introduce the edit metric (sometimes known as the Damerau--Levenshtein metric) on $\LLL$.

\begin{definition}
Define an \emph{edit} of a word $w = w_1\cdots w_n \in\LLL$ to be a transformation of $w$ by one of the following actions, where $u^j\in\LLL$ are arbitrary words and $a,a'\in A$ are arbitrary symbols.
\begin{enumerate}
\item \emph{Substitution}: $w=u^1au^2 \mapsto w' = u^1a'u^2$.
\item \emph{Insertion}: $w=u^1u^2 \mapsto w' = u^1a'u^2$.
\item \emph{Deletion}: $w=u^1au^2\mapsto w' = u^1u^2$.
\end{enumerate}
Given $v,w\in\LLL$, define the \emph{edit distance} between $v$ and $w$ to be the minimum number of edits required to transform the word $v$ into the word $w$: we will denote this by $\hd(v,w)$.
\end{definition}
The following lemma about the size of balls in the edit metric will be crucial for our entropy estimates.
\begin{lemma}\label{lem:edit-balls}
There is $C>0$ such that given $n\in \NN$, $w\in\LLL_n$, and $\delta>0$, we have
\begin{equation}\label{eqn:edit-ball}
\#\{v\in\LLL\mid \hd(v,w)\leq \delta n \} \leq C n^C \left( e^{C\delta} e^{-\delta\log\delta}\right)^n.
\end{equation}
\end{lemma}
Now we can introduce our key definition, which requires that any word in $\LLL$ can be transformed into a word in $\GGG$ with a relatively small number of edits.

\begin{definition} \label{editapproach}
Say that a non-decreasing function $g\colon\NN\to\NN$ is a \emph{mistake function} if $\frac{g(n)}{n}$ converges to $0$.  We say that $\LLL$ is \emph{edit approachable} by $\GGG$, where $\GGG\subset \LLL$,  if there is a mistake function $g$ such that for every $w\in\LLL$, there exists $v\in\GGG$ with $\hat{d}(v,w)\leq g(|w|)$.
\end{definition} 
If $\LLL$ is edit approachable by $\GGG$, and $\GGG$ satisfies the specification property, then it is easy to see that the symbolic space satisfies the following global specification property, which we could call the \emph{almost specification after edits} property:  there exists a mistake function $g$ so that for any words $u, v \in \LLL$, there are words $u', v' \in \LLL$ so that $u'v' \in \LLL$, $\hat{d}(u, u') \leq g(|u|)$ and $\hat{d}(v, v') \leq g(|v|)$.

Edit approachability allows us to replace sufficiently long words in $\LLL$ with words in $\GGG$ in such a way that estimates on Birkhoff averages, and thus estimates on empirical measures, can be well controlled, while at the same time, \eqref{eqn:edit-ball} guarantees that not much entropy is lost this way.  

Control on the Birkhoff averages is given by the following lemma. 

\begin{lemma}\label{lem:stat-near}
For any continuous function $\ph\colon X\to \RR$ and any mistake function $g(n)$, there is a sequence of positive numbers $\delta_n\to 0$ such that if $x,y\in X$ and $m,n\in \NN$ are such that $\hat d(x_1\cdots x_n,\, y_1\cdots y_m)\leq g(n)$, then $|\frac 1n S_n\ph(x) - \frac 1m S_m\ph(y)| \leq \delta_n$.
\end{lemma} 
Although we do not use it in this paper, we prove the following consequence of Lemma \ref{lem:edit-balls} and \ref{lem:stat-near} in \S \ref{sec:lemmas}.

\begin{proposition}\label{prop:full-pressure}
If $\LLL$ is edit approachable by $\GGG$, then $P(\GGG,\ph) = P(\ph)$ for every $\ph\in C(X)$.
\end{proposition}


\subsection{Hamming Approachability} \label{sec:Hamming}
For $v, w \in \LLL_n$, let $d_{\Ham}$ denote the Hamming distance between $v$ and $w$. This is the number of substitutions it takes to transform $v$ into $w$; insertions and deletions are not allowed, and in particular $v,w$ must have the same length.

\begin{definition}
We say that $\LLL$ is \emph{Hamming approachable} by $\GGG$, where $\GGG\subset \LLL$,  if there is a mistake function $g$ such that for every $w\in\LLL$, there exists $v\in\GGG$ with $d_{\Ham}(v,w)\leq g(|w|)$.
\end{definition}

Clearly, if $\LLL$ is Hamming approachable by $\GGG$, then $\LLL$ is edit approachable by $\GGG$. If $\LLL$ is Hamming approachable by $\GGG$, and $\GGG$ satisfies (S)-specification, then it is easy to see that the symbolic space satisfies the \emph{almost specification} property of \cite[\S3.3]{CT}:  there exists a mistake function $g$ so that for any words $u, v \in \LLL$, there are words $u', v' \in \LLL$ so that $u'v' \in \LLL$, $d_{\Ham}(u, u') \leq g(|u|)$ and $d_{\Ham}(v, v') \leq g(|v|)$.  
\subsection{Gibbs properties}
The standard Gibbs property for a shift space says that a measure $m\in\MMM(X)$ is \emph{Gibbs} if there are constants $K,K'>0$ such that
\begin{equation}\label{eqn:Gibbs}
K \leq \frac{m[x_1\cdots x_n]}{e^{-nP(\ph) + S_n\ph(x)}} \leq K'
\end{equation}
for all $x \in X$ and $n \in \NN$. We will require that the upper bound hold uniformly, while the lower bound will only be required to hold when $x_1\cdots x_n\in \GGG$.  More precisely, we make the following definition for a collection $\GGG\subset \LLL$.
\begin{definition}\label{def:Gibbs}
A measure $m\in\MMM(X)$ is \emph{Gibbs for $\ph$ with respect to $\GGG$} if 
there are constants $K,K'>0$ such that
\begin{equation}\label{eqn:upperGibbs}
m[x_1\cdots x_n] \leq K' e^{-nP(\ph) + S_n\ph(x)}
\end{equation}
for every $x\in X$ and $n\in \NN$, and 
\begin{equation}\label{eqn:lowerGibbs}
m[x_1\cdots x_n] \geq K e^{-nP(\ph) + S_n\ph(x)}
\end{equation}
whenever $x\in X$ and $n\in \NN$ are such that $x_1\cdots x_n\in \GGG$.
\end{definition}
Definition \ref{def:Gibbs} is property \ref{A.Gibbs} of Theorem \ref{thm:main}.  Theorem \ref{thm:CT} provides examples of measures satisfying this definition.

\subsection{Properties under factors} \label{sec:factors}

One advantage of our techniques is that they behave well under factors. We let $\Sigma$ be a shift space, and we let $\GGG \subset \LLL(\Sigma)$. Suppose that $X$ is a subshift factor of $\Sigma$, that is, there exists a continuous surjective map
$\pi\colon \Sigma \to X$ such that $\sigma\circ\pi=\pi\circ \sigma$. By \cite[Theorem 6.29]{LM},
$\pi$ is a block code: there exist $r\in\mathbb{N}$ and
$\psi\colon\mathcal{L}_{2r+1}\to \mathcal A$, where $\mathcal A$ is the alphabet of $X$, such that
\begin{equation}\label{eqn:block-code}
(\pi x)_n=\psi(x_{n-r}x_{n-r+1}\cdots x_{n+r-1}x_{n+r}).
\end{equation}
This induces a surjective map $\Psi\colon\mathcal{L}(\Sigma)_{n+2r}\to\mathcal{L}(X)_n$
by
\[
\Psi(w_1\cdots w_{n+2r})=\psi(w_1\cdots w_{2r+1})\psi (w_2\cdots w_{2r+2})
\cdots\psi(w_n\cdots w_{n+2r}).
\]
We set $\tilde{\GGG}=\Psi(\GGG)$. The key to our study of $X$ is that $\tilde \GGG$ inherits a number of good properties of $\GGG$, including in particular \ref{A.spec}, \ref{A.ASD}, and the condition \ref{I} that appears in Theorem \ref{thm:CT} below.

\begin{lemma} \label{Gunderfactors}
Let $\GGG \subset \LLL(\Sigma)$ and $\tilde \GGG \subset \LLL( X)$ be as above.
\begin{enumerate}
\item If $\GGG$ satisfies \ref{A.spec}, then $\tilde \GGG$ satisfies \ref{A.spec}.
\item If $\GGG$ satisfies \ref{A.ASD}, then $\tilde \GGG$ satisfies \ref{A.ASD}.
\item If $\GGG$ satisfies \ref{I}, then $\tilde \GGG$ satisfies \ref{I}.
\end{enumerate}
\end{lemma}

Furthermore, if $\CCC^p \GGG \CCC^s$ is a decomposition for $\LLL(\Sigma)$, there is a natural decomposition for $\LLL(X)$. We define $\tilde \CCC^p$ by taking $\Psi(\CCC^p \LLL_{2k}(\Sigma))$, and  $\tilde \CCC^s$ by taking $\Psi(\LLL_{2k}(\Sigma) \CCC^s)$. It is easy to check the following lemma.
\begin{lemma} \label{factordecomp}
If $\CCC^p \GGG \CCC^s$ is a decomposition for $\LLL(\Sigma)$, then $\tilde \CCC^p \tilde \GGG \tilde \CCC^s$ is a decomposition for $\LLL(X)$. If 
$h(\CCC^p \cup \CCC^s) =0$, then $h(\tilde \CCC^p \cup \tilde \CCC^s) =0$.
\end{lemma}

\section{Consequences of Theorem \ref{thm:main}}\label{sec:consequences}

\subsection{Unique equilibrium states}

The following result from ~\cite{CT,CT2} provides unique equilibrium states which satisfy the weak Gibbs property~\ref{A.Gibbs}, and is our primary tool for finding reference measures to which Theorem~\ref{thm:main} applies. Roughly speaking, Conditions \ref{I}--\ref{II} state that $\CCC^p$ and $\CCC^s$ contain all obstructions to specification (for the system) and regularity (for the potential), while Condition \ref{III} states that these obstructions carry smaller pressure than the whole system.

\begin{theorem}[\cite{CT2}, Theorem C and Remark 2.2] \label{thm:CT}
Let $(X,\sigma)$ be a subshift on a finite alphabet and $\ph\in C(X)$ a potential.  Suppose there exist collections of words $\CCC^p,\GGG,\CCC^s\subset \LLL$ such that $\CCC^p\GGG\CCC^s = \LLL$ and the following conditions hold:
\begin{enumerate}[label=\textbf{(\Roman{*})}]
\item\label{I} $\GGG^M$ has (W)-specification for every $M\in \NN$; \label{item1}
\item\label{II} $\ph$ has the Bowen property on $\GGG$;
\item\label{III} $P(\CCC^p\cup\CCC^s,\ph) < P(\ph)$.
\end{enumerate}
Then $\ph$ has a unique equilibrium state $m_\ph$, and $m_\ph$ is Gibbs for $\ph$ with respect to $\GGG$. In particular, $m_\ph$ satisfies~\ref{A.Gibbs}.
\end{theorem}


We also note that Theorem B of \cite{CT4} is a non-symbolic version of this result whose hypotheses are weaker than those of Theorem \ref{thm:CT}.

\begin{theorem}\label{cor:LDP-small-obstr}
Let $X$ be a subshift on a finite alphabet and $\ph\in C(X)$ a potential.  Suppose $\LLL$ has a decomposition $\LLL = \CCC^p\GGG\CCC^s$ satisfying  \ref{A.ASD} and \ref{I}--\ref{III}.  Then writing $m_\ph$ for the unique equilibrium state of $\ph$, the system $(X,\sigma)$ satisfies a level-2 large deviation principle with reference measure $m_\ph$ and rate function $q^\ph$ given by~\eqref{eqn:rate}.
\end{theorem}
\begin{proof}
Condition  \ref{I} implies \ref{A.spec}.  By Theorem~\ref{thm:CT} there is a unique equilibrium state $m_\ph$, and moreover $m_\ph$ satisfies \ref{A.Gibbs}.  Thus, Theorem \ref{thm:main} gives the result.
\end{proof}

For a shift space $X$ and a collection of words $\CCC\subset \LLL$, it is typically much easier to verify $h(\CCC) <h(X)$ than $P(\CCC,\ph)<P(X, \ph)$.  For $\beta$-shifts, 
it was shown in \cite[Proposition 3.1]{CT2} that $\ref{III}$ holds for every Bowen potential $\ph$, and we show in \S \ref{sec:apps} that this is also true for $S$-gap shifts.
However, for other shift spaces where no analogous argument is available yet, the following lemma is a convenient way to ensure that \ref{III} holds for a large class of functions.
\begin{lemma} \label{lem:BR}
Suppose $X$ is a shift space and $\CCC \subset \LLL(X)$ is a collection of words such that $h(\CCC) < h(X)$, and let $\ph\colon X \to \RR$. If $\ph$ satifies the bounded range condition
\begin{equation} \tag{BR} \label{BR}
\sup\varphi-\inf\varphi<h(X) - h(\CCC),
\end{equation}
then $P(\CCC,\ph) < P(X, \ph)$.
\end{lemma}
Often we can take $h(\CCC)=0$, in which case the condition \eqref{BR} on $\ph$ reduces to the condition
\begin{equation} \tag{BR$_0$} \label{BR0}
\sup\varphi-\inf\varphi<h(X).
\end{equation}

\subsection{Factors}

Our results are well behaved under the operation of passing to a subshift factor.
\begin{theorem} \label{factors}
Let $\Sigma$ be a subshift on a finite alphabet, and suppose that $\LLL = \LLL(\Sigma)$ has a decomposition $\LLL = \CCC^p\GGG\CCC^s$ satisfying \ref{A.ASD} and \ref{I}. Assume further that $h(\CCC) =0$, where $\CCC = \CCC^p \cup \CCC^s$. Let $X$ be a subshift factor of $\Sigma$ and $\ph \colon X \to\RR$ be a continuous function satisfying \eqref{BR0} and the Bowen property.
Then $\ph$ has a unique equilibrium state $m_\ph$, and the system $(X,\sigma)$ satisfies a level-2 large deviation principle with reference measure $m_\ph$ and rate function $q^\ph$ given by~\eqref{eqn:rate}.
\end{theorem}
\begin{proof}
By Lemmas \ref{Gunderfactors} and \ref{factordecomp}, the language $\tilde\LLL$ of $X$ has a decomposition $\tilde\CCC^p\tilde\GGG\tilde\CCC^s$ such that $\tilde\GGG$ satisfies \ref{I} and \ref{A.ASD}, and $h(\tilde\CCC^p \cup \tilde \CCC^s) = 0$.  The Bowen property for $\ph$ implies \ref{II}, and Lemma \ref{lem:BR} gives \ref{III}.  Thus Theorem \ref{cor:LDP-small-obstr} gives the result.
\end{proof}

%
%
%
%
%
%
%
%
%
%

\subsection{$\beta$-shifts and $S$-gap shifts}

Our main examples are the $\beta$-shifts, the $S$-gap shifts, and their subshift factors.   For all of these examples we can take $h(\CCC)=0$, and for the $\beta$-shifts (in \cite[Proposition 3.1]{CT2}) and $S$-gap shifts (in \S \ref{sec:Sgap}), we can show that every Bowen potential satisfies
\begin{equation}\label{eqn:hyperbolic}
P(\ph) > \sup_{\mu\in \MMM_\sigma(X)} \int\ph\,d\mu,
\end{equation}
which in turn implies $P(\CCC,\ph)<P(\ph)$ and removes the need for the bounded range condition. For factors of $\beta$-shifts and $S$-gap shifts, we do require the additional assumption \eqref{BR0} on the potential $\ph$ at present. 

\subsubsection{$S$-gap shifts} 

An $S$-gap shift $\Sigma_S$ is a subshift of $\{0,1\}^{\mathbb{Z}}$ defined by the rule
that for a fixed $S\subset\{0,1,2,\cdots\}$, the number of $0$'s between consecutive
$1$'s is an integer in $S$. That is, the language of $\Sigma_S$ is
\[
\{0^n10^{n_1}10^{n_2}1\cdots 10^{n_k}10^m \mid n_i\in S\text{ for all }1\le i\le k\text{ and }n,
m\in\mathbb{N}\},
\]
together with $\{0^n\mid n\in\mathbb{N}\}$, where we assume that $S$ is infinite (when S is finite, $\Sigma_S$ is sofic and can be analysed without the techniques of this paper).   The language for $\Sigma_S$ admits the following decomposition:
\begin{align*}
\mathcal{G}
&=
\{0^{n_1}10^{n_2}10^{n_3}1\cdots 10^{n_k}1\mid n_i\in S \text{ for all }1\le i\le k\},\\
\mathcal{C}^p
&=
\{0^n1\mid n\not\in S\},\\
\mathcal{C}^s
&=
\{0^n\mid n\in\mathbb{N}\},
\end{align*}
which was first studied in \cite{CT}.  We verify in \S \ref{sec:Sgap} that this decomposition satisfies Conditions \ref{A.ASD} and \ref{I}--\ref{III} for every Bowen potential $\ph$.

\subsubsection{$\beta$-shifts}
Fix $\beta>1$, write $b=\lceil\beta \rceil$, and let $\omega^{\beta}\in
\{0,1,\cdots,b-1\}^{\mathbb{N}}$ be the greedy $\beta$-expansion of $1$.
Then $\omega^{\beta}$ satisfies
$
\sum_{j=1}^{\infty}\omega_j^{\beta}\beta^{-j}=1,
$
and has the property that $\sigma^j(\omega^{\beta})\preceq\omega^{\beta}$
for all $j\ge 1$, where $\preceq$ denotes the lexicographic ordering.
The $\beta$-shift is defined by
\[
\Sigma_{\beta}=\left\{x\in\{0,1,\cdots,b-1\}^{\mathbb{N}}\mid \sigma^j(x)\preceq\omega^{\beta}
\text{ for all }j\ge 1 \right\}.
\]
The first and second author showed in \cite{CT,CT2} that the language for $\Sigma_{\beta}$
admits a decomposition $\mathcal{L}(\Sigma_{\beta})=\mathcal{G}\mathcal{C}^s$ that satisfies \ref{I}--\ref{III} for every Bowen potential $\ph$.  In \S \ref{sec:beta} we briefly review the construction and show that condition \ref{A.ASD} is also satisfied.




\subsubsection{Results for examples}
We collect our results as applied to these examples in the following theorem. 
We say that a subshift is non-trivial if it is not a single periodic orbit. We proved in \cite[Proposition 2.4]{CT} that a non-trivial subshift factor of a $\beta$-shift or $S$-gap shift has positive entropy. 

\begin{theorem} \label{examples} 
Let $X$ and $\ph$ be one of the following:
\begin{enumerate}
\item $X$ is a $\beta$-shift or an $S$-gap shift, and $\ph$ has the Bowen property;
\item $X$ is a non-trivial subshift factor of a $\beta$-shift or an $S$-gap shift, and $\ph$ satisfies \eqref{BR0} and the Bowen property.
\end{enumerate}
Then $\ph$ has a unique equilibrium state $m_\ph$, and $(X,\sigma)$ satisfies a level 2 large deviation principle with reference measure $m_\ph$ and rate function $q^\ph\colon \MMM(X)\to [-\infty,0]$ given by \eqref{eqn:rate}.
\end{theorem}
\begin{proof}
The case when $X$ is an $S$-gap shift is proved in \S\ref{sec:Sgap}.  The case when $X$ is a $\beta$-shift is proved in \S\ref{sec:beta}.  The result for factors of $\beta$-shifts and $S$-gap shifts follows from Theorem \ref{factors}.
\end{proof}

To the best of our knowledge, the above statement was only previously known in the case when $X$ is a $\beta$-shift and $m_0$ is the measure of maximal entropy \cite{PfS} (apart from the exceptional set of special cases above where $X$ has specification, see \cite{BM, AJ}). 

\subsection{Horseshoe theorem: precise statement}

We now state a more precise version of our `horseshoe result', which is a key step in the proof of Theorem~\ref{thm:main}, and may be of independent interest.

\begin{proposition}\label{prop:ent-appr}
Let $X$ be a shift space and suppose that $\GGG\subset \LLL$ satisfies~\ref{A.spec} and~\ref{A.ASD}.
Then there exists an increasing sequence $\{X_n\}$ of compact $\sigma$-invariant subsets of $X$ with the following properties.
\begin{enumerate}
\item Each $X_n$ is a topologically transitive sofic shift.
\item\label{eqn:FtoG} There is $T\in \NN$ such that for every $n$ and every $w\in \LLL(X_n)$, there are $u,v\in \LLL$ with $|u|,|v|\leq n+T$ such that $uwv\in \GGG$.  
\item Every invariant measure on $X$ is entropy approachable by ergodic measures on $X_n$:
for any $\eta>0$, any $\mu\in \Ms(X)$, and any neighborhood $U$ of $\mu$ in $\Ms(X)$, there exist $n\ge 1$ and $\mu'\in\Mse(X_n)\cap U$ such that $h(\mu')>h(\mu)-\eta$ holds.  
\end{enumerate}
\end{proposition}
By the variational principle and the entropy approachability in Proposition \ref{prop:ent-appr}, we have the further result that $\lim_{n \to \infty} h(X_n) = h(X)$, and more generally
\[
P(X, \ph) = \lim_{n\to\infty} P(X_n, \ph) = \sup_{n\in\NN} P(X_n, \ph)
\]
for every $\ph\in C(X)$.  Thus, we can interpret the sets $X_n$ as well behaved `horseshoes' which can be used to approximate the original space $X$, revealing a structure reminiscent of Katok horseshoes \cite{K}. 

In the proof of the main results, we will use the following consequence of the second property in Proposition \ref{prop:ent-appr}:  If a measure $m\in \MMM(X)$ is Gibbs with respect to $\GGG$, then $m$ has the following Gibbs property on the family of subshifts $\{X_n\}$: 
there exist constants $K_n, K' >0$ such that for every $x\in X_n$ and $k \in \NN$, we have
\begin{equation}\label{eqn:XnGibbs}
K_n \leq \frac{m[x_1\cdots x_k]}{e^{-kP(\ph) + S_k\ph(x)}} \leq K'.
\end{equation}
This follows from the fact that $x_1\cdots x_k$ can be extended to a word in $\GGG$ 
by adding a word to each end whose length is bounded by a constant depending only on $n$. 

\section{Proof of Theorems~\ref{thm:main} and \ref{prop:ent-appr0}}\label{sec:pfs}

The large deviations property in Definition~\ref{def:LDP} comprises an upper bound and a lower bound.  We establish the upper bound first, 
 by verifying criteria given by Pfister and Sullivan in \cite{PfS}.


\subsection{Upper bound} \label{sec:ub}

Given $\mu\in \Ms(X)$, let $q^\ph(\mu) = h(\mu) + \int\ph\,d\mu - P(\ph)$, as in~\eqref{eqn:rate}.  We show that for any closed set $F\subset \MMM(X)$, we have
\begin{equation}\label{eqn:upper-bd}
\varlimsup_{n\rightarrow\infty}\frac{1}{n}\log m(\EEE_n^{-1}(F))
\leq \sup_{\mu\in F\cap\mathcal{M}_{\sigma}(X)} q^\ph(\mu).
\end{equation}

Our key tool is the following result of Pfister and Sullivan. 
\begin{theorem}\label{thm:upper-bd}\cite[Theorem 3.2 and Proposition 4.2]{PfS}
Let $(X,\sigma)$ be a subshift, $m\in\mathcal{M}(X)$, $\psi\in C(X)$, and assume that the equation
\begin{equation}\label{eqn:uef}
\varlimsup_{n\rightarrow\infty}\sup_{w\in\mathcal{L}_n}\left(\frac{1}{n}
\log m([w])+\frac{1}{n}\sup_{x\in[w]} S_n\psi (w) \right)\le 0
\end{equation}
holds.
Then 
\[
\varlimsup_{n\rightarrow\infty}\frac{1}{n}\log m(\EEE_n^{-1}(F))
\le \sup_{\mu\in F\cap\mathcal{M}_{\sigma}(X)}\left(h(\mu)-\int\psi d\mu\right).
\]
\end{theorem}

We will apply Theorem~\ref{thm:upper-bd} with $\psi=P(\ph) - \ph$.
The upper Gibbs bound in~\ref{A.Gibbs} (see \eqref{eqn:upperGibbs}) yields a constant $K'$ such that 
\[
m([w]) \leq K'e^{-nP(\varphi)+ S_n\varphi(x)}
\]
for every $x\in[w]$.  Thus
\[
\frac 1n\log m([w]) + \frac 1n \sup_{x\in [w]} S_n(P(\ph) - \ph)(x) \leq \frac 1n\log(K')\to 0,
\]
for every $x\in[w]$. This establishes~\eqref{eqn:uef}, so Theorem~\ref{thm:upper-bd} provides the desired upper bound.

\subsection{Free concatenation and the gluing map} \label{free:concat}

The following result is \cite[Proposition 3.7]{vC15}.
\begin{proposition}\label{prop:free-concat}
If $\GGG$ has (W)-specification (Definition \ref{def:spec}), then there are $r,s\in \GGG$ and $c\in \LLL$ such that writing 
\[
\BBB = \LLL r \cap s \LLL \cap \GGG 
= \{w\in \GGG \mid w_1\cdots w_{\abs{s}} = s \text{ and } w_{\abs{w}-\abs{r}+1} \cdots w_{\abs{w}} = r\}
\]
and $\FFF := c\BBB = \{cw \mid w\in \BBB\}$,
the collection $\FFF\subset \LLL$ has the free concatenation property.  Moreover, a measure $m$ has the Gibbs property for $\ph$ w.r.t.\ $\GGG$ if and only if it has the Gibbs property for $\ph$ w.r.t.\ $\FFF$.
\end{proposition}
\begin{proof}[Sketch of proof]
We outline the main ideas; details are in \cite[\S6.2]{vC15}.  The basic argument is inspired by Bertrand's proof that shifts with specification have a synchronizing word \cite{aB88}. Let $\tau$ be the gap length for the specification property for $\GGG$. Given $r,s\in \GGG$, let $C(r,s) = \{c\in \LLL \mid |c| \leq \tau \text{ and } rcs\in \GGG\}$ be the set of ``short words which connect $r$ to $s$''; this set is finite and non-empty.  If  $u,v\in \LLL$ are such that $r'=ur \in \GGG$ and $s'=sv \in \GGG$, then $C(r',s') \subset C(r,s)$.  If this inclusion is strict for some choice of $u,v$ then replace $r,s$ with $r',s'$.  Iterate this process; since each $C(r,s)$ is finite and non-empty it must terminate, and we obtain $r,s\in \GGG$ such that $C(r',s') = C(r,s)$ for every $r' = ur\in \GGG$ and $s'=sv\in \GGG$.  Pick any $c\in C(r,s)$, then it is not hard to see from the characterisation of $r,s$ that $\FFF$ has the free concatenation property.  Equivalence of the Gibbs properties follows since words in $\FFF$ can be extended to $\GGG$ with a bounded number of symbols, and vice versa.
\end{proof}
\begin{lemma}\label{lem:Fea}
$\LLL$ is edit approachable by $\FFF$.
\end{lemma}
\begin{proof}
Given $w\in \GGG$, by (W)-specification there are $u,v\in \LLL$ with $\abs{u},\abs{v}\leq \tau$ such that $suwvr\in \GGG$, hence $csuwvr\in \FFF$.  Thus every word in $\GGG$ can be turned into a word in $\FFF$ with at most $\abs{c}+\abs{s}+\abs{r}+2\tau$ edits.  Since $\LLL$ is edit approachable by $\GGG$, this suffices.
\end{proof}

Write $\FFF^*$ for the set of all finite sequences $(w^1,\dots,w^m)$ where each $w^i\in \FFF$.  Let $\Phi\colon \FFF^*\to \FFF$ be the concatenation map $(w^1,\dots,w^m) \mapsto w^1\cdots w^m$.  This extends to a map $\FFF^\NN\to \FFF$ in the natural way.  We note that for each $n_1,\dots, n_k\in \NN$, the restriction of $\Phi$ to $\prod_{i=1}^k \FFF_{n_i}$ is clearly injective.

\subsection{Proof of Proposition~\ref{prop:ent-appr}}
We prove Proposition~\ref{prop:ent-appr}, and thus Theorem ~\ref{prop:ent-appr0}. This is crucial for our large deviation lower bounds. 

\vspace{1ex}
\emph{\textbf{Step 0: Definition and basic properties of $X_n$.}}
First, we define the sequence of shift spaces $X_n$ which will meet our requirements. 
Let $\FFF_{\leq n} = \bigcup_{i=0}^n\FFF_i$,  
and consider the set of words
\begin{equation}\label{eqn:Phi}
 \Phi (\FFF_{\leq n}^*) := \bigcup_{m=1}^{\infty} \{ w^1 \cdots w^m \mid w^i \in  \FFF_{\leq n} \text{ for all } 1 \leq i \leq m \}.
\end{equation}
We can turn this set into the language of a shift space by including all subwords to obtain
\begin{equation}\label{eqn:LXn}
\LLL (X_n) := \{ \text{all subwords of elements of } \Phi (\FFF_{\leq n}^*)\}.
\end{equation}
Then, $X_n$ is defined as the shift whose language is $\LLL(X_n)$. 
That $X_n$ is well defined is verified trivially using \cite[Proposition 1.3.4]{LM}.

\begin{lemma} \label{Xn}
The shift space $X_n$ has the following properties.
\begin{enumerate}
\item For every $w\in \LLL(X_n)$, there are $u,v\in \LLL_{\leq n}$ such that 
$uwv\in\FFF$.
\item $X_n$ is a sofic shift and has (W)-specification with gap size $2n$.
\end{enumerate}
\end{lemma}
\begin{proof}
To check the first property claimed for $X_n$, we observe that if $w\in\LLL(X_n)$ is a subword of $w^1\cdots w^m$, then by appending at most $n$ symbols to either end of $w$, we can obtain a word of the form $w^i\cdots w^j\in \FFF$. The (W)-specification property for $X_n$ follows immediately since words in $\FFF$ can be freely concatenated. To see that $X_n$ is sofic, we note that it can be presented by a loop graph, with each (of the finitely many) loops corresponding to a word in $\FFF_{\leq n}^*$. 
\end{proof}

To get the property of $X_n$ claimed in Proposition \ref{prop:ent-appr}\eqref{eqn:FtoG}, we observe that fixing $u\in \GGG$ and putting $T=\abs{u}$, any word in $\FFF$ can be extended to a word in $\GGG$ by adding $u$ to its beginning.

\begin{remark}
If every word in $\LLL$ can be extended to a word in $\GGG$, then it is easy to show that $X = \overline{\bigcup X_n}$.
\end{remark}



The rest of the proof of Proposition~\ref{prop:ent-appr} is an extension of the approach used by Pfister and Sullivan in \cite{PfS2}:
\begin{enumerate}
\item Construct a subshift $Y\subset X_n$ for some $n\ge 1$ such that every $\nu\in\MMM_\sigma(Y)$ is weak*-close to $\mu$.
\item Use edit approachability of $\LLL$ by $\FFF$ to explicitly build a subshift $H\subset Y$ with a rich structure.
\item Show that $H$ (and hence $Y$) has entropy close to $h(\mu)$ by using this structure.
\item Obtain the measures $\mu'$ as maximal entropy measures for $Y$.
\end{enumerate}

In preparation for the above steps, fix $\eta>0$
and use the ergodic decomposition of $\mu$ together with affinity of the entropy map to find $\lambda = \sum_{i=1}^p a_i \mu_i$ such that 
\begin{itemize}
\item the $\mu_i$ are ergodic;
\item the $a_i$ are rational numbers in $[0,1]$ such that $\sum_{i=1}^p a_i=1$;
\item $D(\mu,\lambda)\leq \eta$;
\item $h(\lambda) > h(\mu)-\eta$.
\end{itemize}
Let $h_i = 0$ when $h(\mu_i)=0$, and $\max(0,\, h(\mu_i) - \eta) < h_i < h(\mu_i)$ otherwise.

\begin{definition}
Given $\nu\in\MMM(X)$ and $\zeta>0$, let
\[
\LLL^{\nu,\zeta} := \{w\in\LLL \mid D(\EEE_{|w|}(x),\nu) < \zeta \text{ for all } x\in [w] \}.
\]
\end{definition}

Combining \cite[Propositions 2.1 and 4.1]{PfS}, we have the following.

\begin{lemma}\cite[Propositions 2.1 and 4.1]{PfS}\label{lem:nepssep}
There exists $N\in\NN$ such that for $n\geq N$ and $1\leq i\leq p$, we have $\#\LLL^{\mu_i,\eta}_n \geq e^{nh_i}$.
\end{lemma}

Because $\LLL$ is edit approachable by $\FFF$, there is a mistake function $g$ such that every $w\in \LLL$ has $v\in \FFF$ with $\hd(v,w)\leq g(|w|)$. By Lemma \ref{lem:stat-near}, we can choose $N$ large enough so that, in addition to the cardinality estimates in Lemma~\ref{lem:nepssep}, we have the following property.
\begin{itemize}
\item If $n\geq N$ and $x,y\in X$ are such that $\hat d(x_1\cdots x_n, y_1\cdots y_m) \leq  g(n)$, then $D(\EEE_n(x),\EEE_m(y)) \leq \eta$.
\end{itemize}
Without loss of generality, assume that $0<a_i<1$ for each $i$.  Choose $n$ such that we have $n_i := a_in \in \NN$, $n_i+g(n_i)\le n$, and $n_i\geq N$ for every $i$, and moreover
\begin{equation}\label{eqn:hn}
\frac{n}{n+\sum_{i=1}^pg(n_i)}(h(\lambda)-\eta)\ge h(\lambda)-2\eta.
\end{equation}
To prove the proposition, we will follow the steps listed above to show that there exists $\mu'\in\mathcal{M}_{\sigma}^e(X_n)$
such that $D(\mu,\mu')\le 6\eta$ and $h(\mu')>h(\mu)-4\eta$.

\vspace{1ex}
\emph{\textbf{Step 1: Definition of $Y\subset X_n$.}}
Fix $K\in \NN$ such that $4/K \leq \eta$.  Now let
\begin{equation}\label{eqn:Y}
Y := \{x\in X_n \mid x_t x_{t+1} \cdots x_{t+Kn-1} \in \LLL^{\mu,5\eta}_{Kn} \text{ for all }t\geq 0\}.
\end{equation}
Then $Y\subset X_n$ is compact and $\sigma$-invariant. Moreover, the following holds.
\begin{lemma}\label{ergodic}
We have $D(\mu,\nu)\le 6\eta$ for any $\nu\in\mathcal{M}_{\sigma}^e(Y)$.
\begin{proof}
Since $\nu$ is ergodic, there exists a generic point $x\in Y$, that is,
$\mathcal{E}_m(x)$ converges to $\nu$.
We choose $L$ so that $nK/L\le\eta$ holds, take an arbitrary integer $m\ge L$
and choose integers $s$ and $0\le q<Kn$ so that $m=sKn+q$ holds. Then, using \eqref{eqn:Y} and the inequalities $\frac qm \leq \frac {Kn}L \leq \eta$, we have
\begin{align*}
D(\mathcal{E}_m(x),\mu)
&\le
\sum_{i=0}^{s-1}\frac{Kn}{m}D(\mathcal{E}_{Kn}(\sigma^{iKn}x),\mu)
+\frac{q}{m}D(\mathcal{E}_q(\sigma^{sKn}x),\mu) \\
&\le
5\eta+\eta 
=
6\eta.
\end{align*}
Thus taking $m\rightarrow\infty$, we have the lemma.
\end{proof}
\end{lemma}

\vspace{1ex}
\emph{\textbf{Step 2: Construction of $H$.}}
For brevity of notation we write $\DDD^i = \LLL^{\mu_i,\eta}_{n_i}$.  Extend the definitions of $n_i,\DDD^i,\mu_i,a_i$ to indices $i> p$ by repeating periodically: that is, if $i=pq+r$, $1\leq r\leq p$, then $n_i = n_r$, $\DDD^i = \DDD^r$, $\mu_i = \mu_r$ and $a_i =a_r$.

By the assumption that $\LLL$ is edit approachable by $\FFF$, we can define a map $\phi_\FFF\colon \LLL\to \FFF$ such that $\hd(w,\phi_\FFF(w))\leq g(|w|)$. We extend the map $\Phi\colon \FFF^*\to \FFF$ to a map $\Phi\colon \LLL^*\to \FFF$ by `editing then gluing'. That is, given $(w^1, \ldots, w^n) \in \LLL^*$, we put $\Phieg(w^1, \ldots, w^n) = \phi_\FFF(w^1) \cdots \phi_\FFF(w^n)$. The map $\Phieg$ extends to subsets of $\LLL^\NN$ in the natural way, and we consider it here with the following domain:
\begin{equation}\label{eqn:phi1}
\Phieg \colon \prod_{j=1}^\infty \DDD^j \to X.
\end{equation}
In other words,  given $\ww = \{w^j\} \in \prod_{j=1}^\infty \DDD^j$, let $v^j = \phi_\FFF(w^j) \in \FFF$ and $\Phieg(\ww) = v^1 v^2 \cdots$.  Let $H = \Phieg(\prod_{j=1}^\infty \DDD^j)$. Then we have $H\subset X_n$ since $n_i+g(n_i)\le n$.

A sort of periodicity is built into the definition of the sequences $\Phieg(\ww)$: the word $v^i$ is an approximation of a suitable generic point for the measure $\mu_i$, and the measures $\mu_i$ repeat periodically ($\mu_{i+p}=\mu_i$).  The following lemma states that following $\Phieg(\ww)$ for a single ``cycle'' of this periodic behaviour gives a good approximation to $\mu$.  We write $\ell_j
=\ell_j(\ww) = |v^j|$ for the length of the words associated to the index $j$, and observe that $|\ell_j - n_j| \leq g(n_j) \leq g(n)$.

\begin{lemma}\label{lem:Happrox}
Fix $\ww\in\prod_{j=1}^\infty\DDD^j$.  For $q\geq 0$, let $c_q =c_q(\ww)= \sum_{r=1}^{p} \ell_{qp+r}$ be the length of the $q$th ``cycle'' in $\Phieg(\ww)$ and let $b_m =b_m(\ww)= \sum_{q=0}^{m-1} c_q$.  Then we have $D(\EEE_{c_m}(\sigma^{b_m}\Phieg(\ww)),\mu)\leq 3\eta$.
\end{lemma}
\begin{proof}
Choose $x^j\in [w^j]$ for each $j\in\NN$, so that by the definition of $\DDD^j$, we have $D(\EEE_{n_j}(x^j),\mu_j)\leq \eta$.  Let $y=\sigma^{b_m}\Phieg(\ww)$ and let $d_j = \sum_{i=0}^{j-1} \ell_{mp+i}$ for $1\leq j\leq p$.   By the definition of $\Phieg$ and the property following Lemma~\ref{lem:nepssep}, we have
\[
D(\EEE_{\ell_j}(\sigma^{d_j} y),\mu_j) \leq D(\EEE_{\ell_j}(\sigma^{d_j} y), \EEE_{n_j} (x^{mp+j})) + D(\EEE_{n_j} (x^{mp+j}),\mu_j) \leq 2\eta.
\]
Observe that $c_q\approx n$: more precisely, we have
\begin{equation}\label{eqn:cq}
|c_q - n| \leq \sum_{r=1}^p |\ell_{qp + r} - n_r| \leq p g(n).
\end{equation}
Taking convex combinations gives
\begin{align*}
D(\mu,\EEE_{c_m}(y)) &\leq D\left(\sum_{j=1}^p a_j \mu_j,\,  \sum_{j=1}^p \frac {\ell_j}{c_m} \EEE_{\ell_j}(\sigma^{d_j} y)\right) + \eta \\
&\leq \left(\sum_{j=1}^p \left| a_j - \frac{\ell_j}{c_m} \right|\right) + 2\eta \leq 3\eta,
\end{align*}
provided $N$ is chosen large enough such that $n\geq n_j\geq N$, and such that \eqref{eqn:cq} guarantees we have $\sum_{j=1}^p |a_j - \frac{\ell_j}{c_m}| \leq \eta$.
\end{proof}

We are now in a position to show that $H\subset Y$.  Given $y = \Phieg(\ww) \in H$ and $t\in\NN$, we can choose $m_1,m_2$ such that
\[
b_{m_1-1} \leq t < b_{m_1} < b_{m_2} \leq t+Kn < b_{m_2+1},
\]
and so
\[
\EEE_{Kn}(\sigma^t y) = \left(\sum_{q=m_1}^{m_2-1} \frac{c_q}{Kn}\EEE_{c_q}(\sigma^{b_q}y)\right) + \xi_1 \EEE_{b_{m_1}-t}(\sigma^ty) + \xi_2 \EEE_{t+Kn-bm_2}(\sigma^{b_{m_2}}y),
\]
where $0\leq \xi_1,\xi_2 \leq \frac{n+pg(n)}{Kn} \leq \eta$.  Each of the empirical measures in the large sum is within $3\eta$ of $\mu$, by Lemma~\ref{lem:Happrox}, and thus we have
\[
D(\EEE_{Kn}(\sigma^t y),\mu) \leq 5\eta.
\]
In particular, this shows that $y\in Y$.

\vspace{1ex}
\emph{\textbf{Step 3. Estimation of entropy of $H$.}}
Now we use the definition of $H$ to estimate its topological entropy. Our key tool will be the estimate obtained in Lemmas \ref{lem:edit-balls}. 

\begin{lemma}\label{lem:Hentropy}
The topological entropy of $H$ is at least $h(\mu) - 4\eta$.
\end{lemma}
\begin{proof}
Fix $m\in\NN$ and set $b'=n+\sum_{j=1}^pg(n_j)$. Note that
\begin{equation}
mb'\ge \sup_{\ww\in \prod_{j=1}^\infty \DDD^j }b_m(\ww)\text{ and }\frac{n}{b'}(h(\lambda)-\eta)\ge h(\lambda)-2\eta
\end{equation}
holds, where $b_m(\ww)$ is as in Lemma \ref{lem:Happrox}.  Moreover, since $n=\sum_{j=1}^p n_j$ and each $n_j\geq N$, we have $b' \geq pN$.

Let $\zz\in \prod_{j=1}^\infty \DDD^j$ be arbitrary,
and given $\ww\in \prod_{j=1}^{mp} \DDD^j$, let $\ww\zz$ denote the concatenation of $\ww$ and $\zz$, so that $(\ww\zz)^j = w^j$ if $1\leq j \leq mp$ and $z^{j-mp}$ otherwise.

Let $\phi_m\colon \prod_{j=1}^{mp} \DDD^j \to \LLL_{mb'}$ be the map that takes $w^1,\dots,w^{mp}$ to the first $mb'$ symbols of $\Phieg(\ww\zz)$, where 
$\Phieg$ is the `edit and glue' map from Step 2.  Note that $\phi_m(\prod_{j=1}^{mp} \DDD^j) \subset \LLL_{mb'}(H)$.

Now in order to estimate the entropy of $H$, we will use our estimates on the cardinality of $\DDD^j$ together with a bound on $\#\phi_m^{-1}(v)$ for $v\in \LLL_{mb'}$. 
Recall that $\phi_\FFF\colon \LLL\to \FFF$ is a map which satisfies $\hd(w,\phi_\FFF(w))\leq g(|w|)$. First we use Lemma~\ref{lem:edit-balls}, recalling that 
$\frac{g(n)}{n}\to 0$, to fix $N_0$ sufficiently large so that
\begin{equation}\label{eqn:phiGinv}
\#\{w\in\LLL \mid \phi_\FFF(w)=v\} \leq e^{\eta |v|/2}
\end{equation}
for every $v\in\FFF$ with $|v|\geq N_0$.

Now as $w$ ranges over $\DDD^j$, the word $\phi_\FFF(w)$ may vary in length; however, since its $\hat{d}$-distance from $w$ is at most $g(|w|)$, the number of different lengths it can take is at most $2g(|w|)+1$.  As above, given $\ww \in \prod_{j=1}^{mp} \DDD^j$ we write $\ell_j = \ell_j(\ww) = |\phi_\FFF(w^j)|$, so $\ell_j \in [n_j - g(n_j), n_j + g(n_j)]$.

We see that as $\ww$ ranges over $\prod_{j=1}^{mp} \DDD^j$, the number of different values taken by $(\ell_1,\dots,\ell_{mp})$ is bounded above by
\begin{equation}\label{eqn:vecn}
(2g(n)+1)^{mp} = e^{mp\log(2g(n)+1)} \leq e^{mb'\frac{\log(2g(n)+1)}{\min_j n_j}}
\leq e^{\eta mb'/2},
\end{equation}
where the last inequality follows from observing that $n_j \approx a_j n$ and choosing $N$ sufficiently large (since each $n_j\geq N$).

Given $u\in \LLL_{mb'}(H)$ and a fixed choice of $(\ell_1,\dots,\ell_{mp})$, it follows from \eqref{eqn:phiGinv} that the number of $\ww\in \prod_{j=1}^{mp} \DDD^j$ with $\phi_m(\ww) = v$ and $\ell_j(\ww)=\ell_j$ for each $1\leq j\leq mp$ is at most
\[
\prod_{j=1}^{mp} e^{\eta \ell_j/2} = e^{\frac \eta 2 \sum_{j=1}^{mp} \ell_j}
 \leq e^{\frac \eta 2 mb'}.
\]
Combining this with \eqref{eqn:vecn}, we see that
$
\#\phi_m^{-1}(v) \leq e^{\eta mb'},
$
and thus we obtain the estimate
\[
\#\LLL_{mb'}(H) \geq e^{-\eta mb'} \prod_{j=1}^{mp} (\#\DDD^j).
\]
Using Lemma~\ref{lem:nepssep}, it follows that
\begin{align*}
h(H) &\geq \left(\lim_{m\to\infty} \frac 1{mb'} \sum_{j=1}^{mp} \log\#\DDD^j \right) - \eta \\
&\geq \left(\lim_{m\to\infty} \frac 1{mb'} \sum_{j=1}^{mp} n_j h_j \right) - \eta 
\geq h(\lambda) - 3\eta\geq h(\mu)-4\eta.\qedhere
\end{align*}
\end{proof}

\vspace{1ex}
\emph{\textbf{Step 4: End of the proof of Proposition \ref{prop:ent-appr}.}}
Let $\mu'$ be an ergodic measure of maximal entropy for $Y$.
Lemma \ref{ergodic} shows that $D(\mu',\mu)\leq 6\eta$, and Lemma \ref{lem:Hentropy} shows that $h(\mu') \geq h(\mu)-4\eta$.  Since $Y \subset X_n$ by definition, this completes the proof of Proposition~\ref{prop:ent-appr}.

\subsection{Lower bounds}

Now we complete the proof of Theorem~\ref{thm:main} by showing that the lower bound
\begin{equation}
\label{lowereq}
\varliminf_{n\rightarrow\infty}\frac{1}{n}\log m\left(\{x\in X:\mathcal{E}_n(x)\in U\}\right)\ge \sup_{\mu\in U}q^{\varphi}(\mu)
\end{equation}
holds for any open set $U\subset\mathcal{M}(X)$, where $q^{\varphi}(\mu)$ is as in~\eqref{eqn:rate}.

To show (\ref{lowereq}), it is sufficient to show that
for any $\mu\in\mathcal{M}(X)$ and any open neighborhood $U\subset\mathcal{M}(X)$ of $\mu$,
\begin{equation}
\label{lower3}
\varliminf_{n\rightarrow\infty}\frac{1}{n}\log m\left(\{x\in X:\mathcal{E}_n(x)\in U\}\right)\ge q^{\varphi}(\mu).
\end{equation}
If $\mu$ is not $\sigma$-invariant, then $q^{\varphi}(\mu)=-\infty$ and so the equation (\ref{lower3})
is trivial. Thus, we will prove the equation (\ref{lower3}) for $\mu\in\mathcal{M}_{\sigma}(X)$.

Let $\mu\in\mathcal{M}_{\sigma}(X)$ and $\eta>0$. Then by Proposition \ref{prop:ent-appr}, there exists an ergodic measure
$\nu\in U \cap \mathcal{M}_{\sigma}^e(X_k)$ for some $k$ such that $h(\nu)>h(\mu)-\eta$ and $\int \varphi\, d\nu>\int \varphi\, d\mu-\eta$.  We use $\nu$ to build a subset of $\EEE_n^{-1}(U)$, as follows.

Take $\zeta>0$ so small that $\mathbb{B}(\nu,2\zeta)\subset U$ and every measure $\nu'$ in this neighbourhood has $|\int\ph\,d\nu' - \int\ph\,d\nu|\leq \eta$.  In particular, for every $w\in \LLL^{\nu,\zeta}$, we have $[w]\subset \EEE_n^{-1}(U)$.
Then, again by \cite[Propositions 2.1 and 4.1]{PfS}, for all sufficiently large $n$ we have
\begin{equation}\label{eqn:Lnz}
\# (\LLL_n^{\nu,\zeta}\cap\LLL(X_k)) \geq e^{n(h(\mu)-\eta)}.
\end{equation}

We note that by the Gibbs property \eqref{eqn:XnGibbs}, we have
\[
m[w] \geq K_k e^{-nP(\ph) + S_n\ph(x)}
\]
for all $w\in\LLL(X_k)_n$ and $x\in [w]$.  In particular, when $w\in \LLL^{\nu,\zeta}$ this yields
\[
m[w] \geq K_k e^{-nP(\ph) + n \int\ph\,d\nu - n\eta}
\geq K_k e^{n(-P(\ph) + \int\ph\,d\mu - 2\eta)}.
\]
Using the estimate~\eqref{eqn:Lnz} and the fact that $[w]\subset\EEE_n^{-1}(U)$ for every $w\in\LLL^{\nu,\zeta}$, we obtain
\[
m(\EEE_n^{-1}(U)) \geq K_k e^{n( h(\mu) -P(\ph) + \int\ph\,d\mu - 3\eta)}.
\]
Since $\eta>0$ was arbitrary, this establishes the lower bound~\eqref{lower3}.

\section{Applications}\label{sec:apps}

\subsection{$S$-gap shifts}\label{sec:Sgap}
The family of $S$-gap shifts were introduced in \cite{LM}, and have received a recent increase in attention \cite{AJ, CT, BG}. To check the conditions of Theorem \ref{cor:LDP-small-obstr} for a Bowen potential $\ph$, we verify the specification properties \ref{A.spec} and \ref{I} on $\GGG$ and $\GGG^M$, the edit approachability property \ref{A.ASD}, and the estimate \ref{III} on $P(\CCC^p\cup \CCC^s,\ph)$.

\subsubsection{Specification properties} \label{sec:sgapspec}
It is immediate that $\GGG$ has the free concatenation property, and thus Condition \ref{A.spec} is satisfied.  Condition \ref{I} holds because a word in $\GGG^M$ has the form $0^n10^{n_1}10^{n_2}1\cdots 10^{n_k}10^m$, where $n_i\in S$ for all $1\le i\le k$, and $n,m \leq M$. Thus, any word in $\GGG^M$ can be extended to a word in $\GGG$ by adding a uniformly bounded number of symbols at each end (the number of symbols to be added depends on $M$, but not on the length of the word), and this implies that $\GGG^M$ has the (W)-specification property.

\subsubsection{Edit approachability}

Because $S$ is infinite, we can choose for every $n\in \NN$ some $s_n\in S$ such that $\frac {s_n}n\to 0$ and $s_n\to\infty$.  (Note that the same element of $S$ may appear as $s_n$ for multiple values of $n$.)  Now define $g\colon \NN\to\NN$ by $g(n) := 2(\lceil n/s_n\rceil + s_n)$, and observe that $g$ is a mistake function.

Let $z\in\mathcal{L}(X)_n$ and write $s=s_n$. The word $z$ has the form
\[
z =0^k10^{n_1}10^{n_2}\cdots 0^{n_i}10^\ell.
\]
We now change at most $k/s$ of the symbols $0^k$ to form the word
$z^p:=0^i10^s10^s\cdots 0^s 10^s$ $(0\le i\le s)$.
We also change at most $\ell/s$ of the symbols $0^\ell$ to form the word
$z^s:=0^s10^s10^s\cdots 0^s 10^j$ $(0\le j\le s)$.
We set $w:=z^p10^{n_1}10^{n_2}\cdots 0^{n_i}1z^s$,
$u:=0^{s-i}$ and $v:=0^{s-j}1$.
Then we have $\hat{d}(z,uwv)\le 2(\lceil n/s_n\rceil+s_n) =g(n)$, and
$u w v\in \mathcal{G}$ by the definition of $\GGG$. This shows that $\GGG$ satisfies \ref{A.ASD}. 


\begin{remark}
When $\Sigma_S$ is mixing -- that is, when $\gcd(S+1)=1$ -- it is possible to show that $\LLL$ is Hamming approachable by $\GGG$. The idea is to combine the argument for edit approachability with some additional combinatorial estimates.  It immediately follows that mixing $S$-gap shifts have the almost specification property.\footnote{In \cite{CT}, the first two authors claimed to give an example of an $S$-gap shift without the almost specification property. That example was in error - there was an elementary mistake in the computation, and the argument here does in fact yield the almost specification property for that example.} It is also not hard to show the following spectral decomposition result: if $d=\gcd(S+1) > 1$, then $(\Sigma_S, \sigma^d)$ is topologically conjugate to a  union of $d$ disjoint mixing $S$-gap shifts. 
\end{remark}

\subsubsection{Estimating $P(\CCC^p\cup\CCC^s,\ph)$}

Now we show that if $\ph$ is any potential with the Bowen property on an $S$-gap shift, then $P(\CCC^p \cup \CCC^s,\ph) < P(\ph)$, verifying Condition \ref{III}.  It is easy to see that $h(\CCC^p \cup \CCC^s) = 0$, so it suffices to show that
\begin{equation}\label{eqn:gap}
P(\ph) > \ulim_{n\to\infty} \sup_{x\in X} \frac 1n S_n\ph(x),
\end{equation}
which is equivalent to every equilibrium state for $\ph$ having positive entropy. Our  strategy is to produce a large number of admissible words that are close (in the edit metric) to a given word, so that no single word can carry full pressure. This strategy was also used to establish \eqref{eqn:gap} for $\beta$-shifts in \cite[Proposition 3.1]{CT2}.  
For $S$-gap shifts, we must deal with a difficulty which does not occur for $\beta$-shifts:  if $x\in \Sigma_S$ is such that positions $i$ and $j$ both admit edits yielding new words $x',x''\in \Sigma_S$, it may not be possible to make both edits simultaneously. This lack of independence between the possible edits means that it is more difficult to produce nearby words than in the case of $\beta$-shifts. Here, we state a sequence of lemmas which prove \eqref{eqn:gap}, 
whose proofs are given in \S \ref{sec:lemmas}.
\begin{lemma}\label{lem:zero}
We have $P(\ph) > \ph(0)$.
\end{lemma}
In the following lemma, we use Lemma \ref{lem:zero} to control words which have a small frequency of occurence of the symbol $1$.
\begin{lemma} \label{lem:rare}
 There exists $\epsilon>0$ and a constant $L=L(\epsilon)$ so that if $x_1\cdots x_n$ contains fewer than $\epsilon n$ occurrences of the symbol $1$, then 
 $\frac 1n S_n\ph(x) \leq \ph(0) + L < P(\ph) - L.$
\end{lemma}
We now control words which do not have a small frequency of occurence of the symbol $1$. This is where we use our strategy of creating a large number of new words by making edits.  We need the following estimate, which is a consequence of Stirling's formula.
\begin{lemma}\label{eqn:stirling3} 
If $\delta n \leq k \leq \frac n2$, then $\log {n\choose k} \geq -n\delta\log\delta - 2\log n$.
\end{lemma}
This estimate can be used to give a lower bound on the cardinality of a set of words where we can control the Birkhoff averages of $\ph$, and we can use this to estimate the pressure from below.

\begin{lemma} \label{lem:biggerthan}
Given $\epsilon$ as in Lemma \ref{lem:rare}, there exists $L'>0$ such that whenever $n$ is sufficiently large and $x_1\cdots x_n\in \LLL$ contains $m\geq \epsilon n$ occurrences of the symbol $1$, we have $\frac{1}{n}S_n \ph (x)  < P(\ph)- L'$.
\end{lemma}

We conclude from Lemma \ref{lem:rare} and Lemma \ref{lem:biggerthan} that
\[
\ulim_{n\to\infty}\sup_{x \in X} \frac 1n S_n\ph(x) \leq \max \{P(\ph) - L, P(\ph)- L'\} <P(\ph),
\]
and it is easy to verify \ref{III} from this together with  $h(\CCC^p \cup \CCC^s) = 0$.

\subsection{$\beta$-shifts}\label{sec:beta}

Every $\beta$-shift can be presented by a countable state directed labelled graph
with vertices $v_1,v_2,\cdots$. For every $i\ge 1$, we draw an edge from $v_i$ to $v_{i+1}$,
and label it with the value $\omega_i^{\beta}$. Next, whenever $\omega_i^{\beta}>0$,
for each integer from $0$ to $\omega_i^{\beta}-1$, we draw an edge from $v_i$ to
$v_1$ labelled by that value.

The $\beta$-shift can be characterised as the set of sequences given by the
labels of infinite paths through the directed graph which start at $v_1$.
For our set $\mathcal{G}$, we take the collection of words labelling a path that
begins and ends at the vertex $v_1$. Thus, $\GGG$ automatically satisfies the free concatenation property, and in particular, \ref{A.spec} holds.

Let $\varphi\colon\Sigma^{\beta}\to\mathbb{R}$ be a continuous function satisfying the
Bowen property. It is shown in \cite[\S3.1]{CT2} that
conditions \ref{I}--\ref{III} in Theorem~\ref{thm:CT} hold, so it only remains to check condition \ref{A.ASD}.

We now show that  $\mathcal{L}$ is Hamming approachable (and thus edit approachable) by $\GGG$ 
with mistake function $g\equiv 1$.
Let $z\in\mathcal{L}_n$. We set $j:=\max\{1\le i\le n :z_i\not =0\}$ and define a new word $w\in\LLL_n$ by
\[
w_i=
\begin{cases}
z_i & (1\le i\le n,i\not=j)); \\
z_j-1 & (i=j).
\end{cases}
\]
It is easy to see that $d_{\text{Ham}}(z, w) = \hat{d}(z,w)=1=g(n)$ and $w\in\mathcal{G}$, which implies \ref{A.ASD}.
It follows that $(\Sigma^{\beta},\sigma)$ satisfies the level-2 large deviations principle with reference measure $m_{\ph}$, and rate
function $q^{\ph}$ given by (\ref{eqn:rate}).

\section{Hamming approachability and Weak lower energy functionals} \label{sec:lowerenergy} For completeness, we describe the connection between our hypotheses and the hypotheses of Pfister and Sullivan \cite{PfS} for the large deviations lower bound. As remarked in \S \ref{sec:Hamming}, if $\LLL$ is Hamming approachable by $\GGG$, and $\GGG$ satisfies (S)-specification, then the symbolic space satisfies the almost specification property, and thus the approximate product property of \cite{PfS}.  So, in this setting, lower large deviations follow from \cite{PfS} by finding a lower weak energy function.

 \begin{proposition} \label{lowerenergy}
Let $(X,\sigma)$ be a shift on a finite alphabet, $m$ a Borel probability measure on $X$, and
$\varphi\colon X\to \mathbb{R}$ a continuous function. Let $\mathcal{L}$ be the language of $X$.
Suppose that $\mathcal{G}\subset\mathcal{L}$ is such that
\begin{enumerate}[label=($\diamondsuit$\arabic{*})]
\item\label{low-Gibbs} $m$ has the lower Gibbs property \eqref{eqn:lowerGibbs} for $\varphi$ with respect to $\mathcal{G}$;
\item\label{Hamming} $\LLL$ is Hamming approachable by $\GGG$;
\item the function $\psi:=P(\varphi)-\varphi$ is non-negative.
\end{enumerate}
Then 
$\psi$ is a weak lower energy function in the sense of Pfister and Sullivan \cite[Definition 3.3]{PfS}.
\begin{proof}
By \cite[Proposition 4.3]{PfS}, it is sufficient to show that
for any $\delta>0$, there exists $N$ so that $n\ge N$ implies that for
each $v\in \mathcal{L}_n$ there exists $w\in \mathcal{L}_n$ satisfying
$d_{\Ham}(v,w)\le \delta n$ and
\begin{equation}
\label{wlef}
\frac{1}{n}\log m([w])+\inf_{x\in [w]}S_n\psi(x)\ge -\delta.
\end{equation}
Fix any $\delta>0$ and let $K$ be the constant from the lower Gibbs property \eqref{eqn:lowerGibbs} given by \ref{low-Gibbs}.
Then there is $N_1$ so that $n\ge N_1$
implies $\frac{1}{n}\log K\ge -\delta$.
By \ref{Hamming}, there exists $N_{\delta}\geq N_1$ so that
$n\ge N_{\delta}$ implies that for each $v\in\mathcal{L}_n$, there exists $w\in\mathcal{G}_n$
such that $d_H(v,w)\le \delta n$. 

Fix any $n\ge N_\delta$, $v\in \mathcal{L}_n$.
By condition \ref{Hamming}, there exists 
$w\in \mathcal{G}_n$ such that
$d_n^H(v,w)\le\delta n$. Let $x\in [w]$.
Then it follows from \ref{low-Gibbs} that
\begin{align*}
\tfrac{1}{n}\log m([w])+\tfrac{1}{n}S_n\psi (x)
&=
\tfrac{1}{n}\log m([w])+P(\varphi)-\tfrac{1}{n}S_n\varphi (x)\\
&\ge
\tfrac{1}{n}\log K \ge
-\delta,
\end{align*}
which implies equation (\ref{wlef}).
\end{proof}
\end{proposition}
\begin{remark}

Given Theorem \ref{thm:CT}, \cite[Propositions 3.1 and 3.2]{PfS}, and \S \ref{sec:apps}, Proposition \ref{lowerenergy} gives another approach to the lower large deviations bound for $\beta$-shifts and $S$-gap shifts. In particular, Proposition \ref{lowerenergy} illuminates the mechanism that is implicitly used in the work of Pfister and Sullivan to obtain large deviations for the measure of maximal entropy for the $\beta$-transformation - the proof of existence of the weak lower energy function is presented in an ad hoc way in their study.
\end{remark}

%

\section{Proofs of Lemmas} \label{sec:lemmas}

\begin{proof}[Proof of Lemma \ref{lem:edit-balls}]

We obtain an upper bound on the number of words that can be obtained by making at most $m$ edits to $w$ as follows.
We introduce an additional symbol $e$ (for `edit'), 
 and construct a new word $w'$ of length $n+m$ which contains exactly $m$ of the symbols $e$, and so that $w_1=w'_1$. 
Note that $\smatrix{n+m \\ n}$ is an upper bound on the number of such words $w'$. 
Now obtain a new word $v \in \LLL$ from $w'$ by performing exactly one of the following actions at each symbol $e$, and then deleting the $e$.
\begin{enumerate}
\item Change the symbol immediately before $e$ to a different symbol.
\item Insert a symbol immediately before $e$.
\item Delete the symbol immediately before $e$.
\item Leave the symbol immediately before $e$ unchanged.
\end{enumerate}
Note that every word $v$ which satisfies $\hd(v,w) \leq m$ can be produced by this procedure. At each symbol $e$, there are a total of $2\#A + 2$ possible actions, so we see that 
\[
\#\{v \mid \hd(v, w) \} \leq (2\#A + 2)^m \smatrix{n+m \\ n}.
\]
From Stirling's formula there is a constant $C'$ such that 
\[
|\log n! - (n\log n - n)| \leq C'\log n
\]
for every $n\in\NN$, and so when $m\leq \delta n$ we have
\begin{align*}
\log \smatrix{n+m \\ n} &=
\log (n+m)! - \log n! - \log m! \\
&\leq \big((n+m)\log(n+m) - n\log n - m\log m\big) \\
&\qquad\qquad + C'(\log (m+n) + \log m + \log n) \\
&= \left(n \log\frac{n+m}{n} + m\log\frac{n+m}{m}\right) + 3C'\log(m+n) \\
&\leq n\left(\log (1+\delta) + \delta \log (1+\delta^{-1})\right) + 3C'\log ((1+\delta)n) \\
&=n \left( (1+\delta)\log(1+\delta) - \delta\log\delta \right) + 3C'\log((1+\delta)n).
\end{align*}
Using the inequalities $1+\delta\leq 2$ and $\log(1+\delta)\leq \delta$, we see that the left-hand side of \eqref{eqn:edit-ball} admits the bound
\begin{align*}
\#\{v\in\LLL&\mid \hd(v,w)\leq \delta n \} \\
&\leq (2\#A+2)^{\delta n}e^{n((1+\delta)\log(1+\delta) -\delta\log\delta)} e^{3C'\log((1+\delta)n)} \\
&\leq (2\#A+2)^{\delta n}e^{2n\delta} e^{n(-\delta\log\delta)} (1+\delta)^{3C'} n^{3C'},
\end{align*}
which completes the proof.
\end{proof}

\begin{proof}[Proof of Lemma \ref{lem:stat-near}] 
Let $\hat g(n) = g(n)+1$, so that $\hat g$ is also a mistake function.  Take $x,y$ and $m,n$ as in the hypothesis of the lemma, and let $k=\hat d(x_1\cdots x_n,\, y_1\cdots y_m) \leq g(n)$.  

Following the set-up of the proof of the previous lemma, we obtain a new word $w'$ by inserting the symbol $e$ into $k$ positions of 
$x_1\cdots x_n$ to mark where an insertion, deletion or substitution will take place to obtain $y_1\cdots y_m$. We write $w=w^1w^2 \cdots w^{k+1}$ so that the last symbol of each $w^i$ with $1 \leq i \leq k$ is $e$ (note that $w^{k+1}$ may be the empty word). Let $w^i_r$ be the word obtained by omitting the last two symbols from $w^i$, and form the word $w_r = w^1_r w^2_r \cdots w^{k+1}_r$ (where $r$ stands for `reduced', and if $|w^i| \leq 2$, then $w^i_r$ is the empty word). 
For $n \geq 0$, let
\begin{multline*}
V(n) = \sup \{ |S_{m'}\ph(x) - S_m\ph(y)| \mid x_1\cdots  x_n = y_1 \cdots y_n \\ \text{ and }  m,m' \in \{n, n+1, n+2 \}   \}.
\end{multline*}
Note that continuity of $\ph$ implies that $\frac 1n V(n) \to 0$.  In particular, for $z \geq 1$, we may write $\epsilon(z) = \sup_{m\geq z} \frac 1m V(m)$ and obtain $\epsilon(z)\to 0$. We will use this fact for ``long'' words, while for ``short'' words we will use the bound $V(n) \leq 2(n+2)\|\ph\| \leq 4(n+1)\|\ph\|$.

Both $x_1\cdots x_n$ and $y_1\cdots y_m$ can be obtained from $w_r$ by inserting at most two symbols at the end of each $w^i_r$, and so
$|S_n\ph(x) - S_m\ph(y)| \leq \sum_{j=1}^{k+1} V(n_j),$
where $n_j=|w^j_r|$. To bound this sum, we let $C_n = \sqrt{\frac{n}{\hat g(n)}}$ and break the sum into two parts, corresponding to $n_j < C_n$ and $n_j\geq C_n$. 
We have
\begin{multline}\label{eqn:Sndiff2}
|S_n\ph(x) - S_m\ph(y)| \leq \sum_{n_j < C_n} V(n_j)
+ \sum_{n_j \geq C_n} V(n_j) \\
\leq \sum_{n_j < C_n} 4(n_j+1)\|\ph\| + \sum_{n_j \geq C_n} n_j \epsilon(C_n) 
\leq 4C_n\|\ph\| \hat g(n) + n \epsilon(C_n),
\end{multline}
where the last inequality uses the fact that there are $k+1\leq \hat g(n)$ values of $j$ in total, and that $\sum n_j\leq n$. 

Now we can estimate the difference in Lemma \ref{lem:stat-near} as
\begin{align*}
\Big\lvert \frac 1n S_n\ph(x) &- \frac 1m S_m\ph(y) \Big\rvert
\leq \frac 1n \abs{S_n\ph(x) - S_m\ph(y)} + \abs{\frac 1n - \frac 1m}|S_m\ph(y)| \\
&\leq  4\|\ph\| C_n \frac{\hat g(n)}{n} + \epsilon(C_n) + \frac{|m-n|}{n} \frac 1m|S_m\ph(y)| \\
&\leq 4\|\ph\|\sqrt{\frac{\hat g(n)}{n}} + \epsilon(C_n) + \frac{\hat g(n)}{n} \|\ph\|.
\end{align*}
Because $\hat g$ is a mistake function, the first and third terms go to $0$ as $n\to\infty$, while $C_n\to\infty$ and so the second term goes to $0$ as well.  This completes the proof of Lemma \ref{lem:stat-near}.
\end{proof}

\begin{proof}[Proof of Proposition \ref{prop:full-pressure}]
Clearly $P(\GGG,\ph)\leq P(\ph)$, so it suffices to prove the other inequality. We compare $\Lambda_n(\LLL,\ph)$ and $\Lambda_n(\GGG,\ph)$ using Lemmas \ref{lem:edit-balls} and \ref{lem:stat-near}.  By edit approachability, for each $w\in \LLL_n$ there exists $v=v(w)\in \GGG$ such that $\hd(v,w)\leq g(|w|)$.  Lemma \ref{lem:edit-balls} tells us that given $v\in\GGG$, the number of words $w\in\LLL_n$ for which $v=v(w)$ is at most
\[
C n^C \left( e^{C\delta} e^{-\delta\log\delta}\right)^n,
\] 
where $\delta = g(n)/n$.  In particular, for all sufficiently large $n$ this expression is bounded above by $
e^{\delta_n' n}$,  where  $\delta_n' \to 0$.

It follows from Lemma \ref{lem:stat-near} that there is $\delta_n\to 0$ such that for every $v,w$ as above and any $x\in [v]$, $y\in [w]$, we have
\[
|S_n\ph(x) - S_{|w|}\ph(y)| \leq n\delta_n.
\]
Together the above estimates imply that 
\[
\Lambda_n(\LLL,\ph) \leq
\sum_{m=n-g(n)}^{n+g(n)} \sum_{w\in \GGG_m} e^{\delta_n' n} e^{n\delta_n + \sup_{y\in [w]} S_m\ph(y)},
\]
and so in particular there is $m\in [n-g(n),n+g(n)]$ such that
\[
\Lambda_m(\GGG,\ph) \geq \frac 1{2g(n)} e^{-(\delta_n' + \delta_n) n} \Lambda_n(\LLL,\ph).
\]
Since $g(n)$ is sublinear and $\delta_n,\delta_n'\to 0$, this implies the result.
\end{proof}

\begin{proof}[Proof of Lemma \ref{Gunderfactors}]
Items 1) and 3) can be obtained by making minor modifications to the proof of Proposition 2.2 in \cite[\S6.2]{CT}, so we omit these arguments and prove only item 2).

Let $\GGG\subset \LLL(\Sigma)$ and $\tilde{\GGG}\subset\LLL(X)$ be as in Lemma \ref{Gunderfactors} and assume that $\GGG$ satisfies \ref{A.ASD}.
Let $g\colon\NN\to\NN$ be a mistake function as in \ref{A.ASD} for $\GGG$.
Then we define a mistake function $\tilde{g}\colon \mathbb{N}\to \mathbb{N}$ by
$\tilde{g}(n)=(4r+3)g(n+2r)+4r$.
Take a $\tilde{z}\in\LLL(X)_n$. Since $\Psi$ is surjective,
there exists $z\in\LLL(\Sigma)_{n+2r}$ so that $\Psi(z)=\tilde{z}$.
Since $\GGG$ satisfies \ref{A.ASD}, we can find $w\in\GGG$ so that
$\hat{d}(z,w)\le g(n+2r)$ holds, where we recall that $r$ is the length of the block code. We set $\tilde{w}=\Psi(w)$.

Because $\hat{d}(z,w)\leq g(n+2r)$, there exist an integer $K\ge n-\big((2r+1)g(n+2r)+2r\big)$
and two increasing sequences $m_1<\cdots<m_K$, $n_1<\cdots<n_K$ so that
\[
z_{m_i-r}\cdots z_{m_i+r}=w_{n_i-r}\cdots w_{n_i+r}
\]
for each $1\le i\le K$.  Because $\Psi$ is a block code with length $r$, we have $\tilde{z}_{m_i}=\tilde{w}_{n_i}$ for $1\le i\le K$.
This implies that
$$\hat{d}(\tilde{z},\tilde{w})\le (n-K)+(|w|-K)\le 2(n-K)+||w|-n|\le \tilde{g}(n).$$
Thus $\GGG$ satisfies \ref{A.ASD}.
\end{proof}

\begin{proof}[Proof of Lemma \ref{lem:BR}]
We have 
\[
\Lambda_n(\CCC,\ph) = \sum_{w\in \CCC_n} e^{\sup_{x\in[w]} S_n\ph(x)} \leq  e^{n(\sup \ph)} \Lambda_n(\CCC,0),
\]
and so $P(\CCC,\ph) \leq h(\CCC) + \sup \ph$. By the variational principle and the assumption \eqref{BR}, we have
\[
 P(X, \ph) \geq h(X) + \inf \ph 
> h(\CCC) + \sup \ph \geq P(\CCC,\ph),
\]
 which proves the lemma.
\end{proof}

\begin{proof}[Proof of Lemma \ref{lem:zero}]
Let $V$ be such that $|S_n\ph(x) - S_n\ph(y)|\leq V$ whenever $x_1\cdots x_n = y_1\cdots y_n$, and in particular $S_n\ph(x) \geq n\ph(0) - V'$ for every $x\in [0^{n-1} 1]$, where $V' = V + \ph(0) - (\inf\ph)$.

Choose $k$ large (just how large will be determined later) and let $n_1,n_2,\dots,n_k\in S$ be distinct.  Let $\pi$ be any permutation of the integers $\{1,\dots,k\}$, and let $w_\pi$ be the word $0^{n_{\pi(1)}} 1 0^{n_{\pi(2)}}1 \cdots 0^{n_{\pi(k)}} 1$ of length $N = \sum_{j=1}^k (n_j + 1)$. 
 The estimates in the previous paragraph give
\[
S_N\ph(y) \geq N\ph(0) - kV'
\]
for every $y\in [w_\pi]$.  Now let $\vec\pi=(\pi_1,\dots,\pi_m)$ be any sequence of $m$ such permutations, and let $v_{\vec\pi} = w_{\pi_1} \cdots w_{\pi_m}$.  Choosing any $y_{\vec\pi}\in[v_{\vec\pi}]$, we obtain the estimate
\begin{equation}\label{eqn:mN}
\Lambda_{mN}(\LLL,\ph) \geq \sum_{\vec\pi} e^{S_{mN}\ph(y_{\vec\pi})}
\geq (k!)^m e^{mN\ph(0) - mkV'}.
\end{equation}
We have the general bound
\begin{equation}\label{eqn:stirling1}
\log(k!) = \sum_{j=1}^k \log j \geq \int_1^k \log t\,dt = k\log k - k - 1,
\end{equation}
which yields
\[
\log \Lambda_{mN}(\LLL,\ph)
\geq m(k\log k - k - 1) + mN\ph(0) - mkV',
\]
so that dividing by $mN$ and sending $m\to\infty$ we have
\[
P(\ph) \geq \ph(0) + \tfrac kN\left(\log k - 1 - \tfrac 1k - V'\right).
\]
Taking $k$ large gives the result.
\end{proof}
\begin{proof}[Proof of Lemma \ref{lem:rare}] By Lemma \ref{lem:zero}, there exists $\epsilon>0$ such that $\ph(0) + 2\epsilon V' < P(\ph)$, where $V'$ is the constant from the proof of the previous lemma.  Note that if $x_1\cdots x_n$ contains fewer than $\epsilon n$ occurrences of the symbol $1$, then
$
S_n\ph(x) \leq n\ph(0) + \epsilon n V',
$
and in particular
\begin{equation}\label{eqn:rare}
\tfrac 1n S_n\ph(x) \leq \ph(0) + \epsilon V' < P(\ph) - \epsilon V'.
\end{equation}
Settig $L=\epsilon V'$ gives the result.
\end{proof}

\begin{proof}[Proof of Lemma \ref{eqn:stirling3}]

We use the upper bound
\[
\begin{aligned}
\log(k!) &= \sum_{j=1}^k \log j \leq \int_1^{k+1}\log t\,dt = (k+1)\log(k+1) - k \\
&= (k\log k - k) + k\log\left(1 + \tfrac 1k\right) + \log(k+1) \\
&\leq (k\log k - k) + (1+ \log(k+1)),
\end{aligned}
\]
which together with \eqref{eqn:stirling1} gives, for all large $n$,
\begin{align*}
\log \begin{pmatrix} n \\ k \end{pmatrix} &=
\log(n!) - \log(k!) - \log((n-k)!) \\
&\geq (n\log n - n - 1) - (k\log k - k) - (1 + \log(k+1)) \\
&\quad - ((n-k)\log(n-k) - (n-k)) - (1 + \log (n-k+1)) \\
&\geq n h\left( \tfrac kn \right) - 2\log n,
\end{align*}
where $h(\delta) = -\delta\log \delta - (1-\delta) \log(1-\delta)$.
\end{proof}

\begin{proof}[Proof of Lemma \ref{lem:biggerthan}]
Now assume that $x_1\cdots x_n$ contains $m\geq \epsilon n$ occurrences of the symbol $1$.  By considering a smaller collection of indices where the entry is $1$ if necessary, we may assume that $m\leq 2\epsilon n$.  

Given $\delta>0$ small (just how small will be determined later), let $\delta m < k < 2\delta m$.  Let $R$ be the set of indices in which $x_1\cdots x_n$ has a nonzero symbol, and let $\mathcal{Z}$ be the collection of subsets of $R$ with exactly $k$ elements.

We define a map $\phi\colon \mathcal{Z}\to X$ as follows.  Fix $n_1\neq n_2\in S$.  Given $Z\in \mathcal{Z}$, at each index $k\in Z$ insert the word $0^{n_1}1$ into $x$, unless $x_{k+1} \cdots x_{k+{n_1}+1} = 0^{n_1} 1$, in which case insert the word $0^{n_2} 1$.  This is allowed by the definition of the $S$-gap shift, and we note that $\phi$ is 1-1.

Let $\ell=\max\{n_1,n_2\} + 1$, and observe that $\phi(Z)$ is obtained from $x$ by inserting at most $k\ell$ symbols, so that if $p$ is the size of the alphabet, then the map $\Phi\colon \mathcal{Z}\to \LLL_n$ obtained by truncating $\phi(Z)$ to the first $n$ symbols has the property that $\#\Phi^{-1}(w) \leq p^{k\ell}$ for each $w\in \LLL_n$.  

We conclude that the map $\Phi$ yields at least $\smatrix{ m \\ k} p^{-k\ell}$ words $w$ in $\LLL_n$ with the property that
\[
S_n\ph(y) \geq S_n\ph(x) - k\ell V' \geq S_n\ph(x) - 4\epsilon\delta n V'
\]
for every $y\in [w]$.  In particular, together with Lemma \ref{eqn:stirling3} and the conditions on $m$ and $k$, this gives the estimate
\begin{multline*}
\log \Lambda_n(\LLL,\ph) \geq -m\delta \log\delta -2\log m -k\ell\log p+ S_n\ph(x) - 4\epsilon\delta n V' \\
\geq (\epsilon n)(-\delta\log\delta) - 4\log(\epsilon n) - 4\epsilon\delta \ell\log p + S_n\ph(x) - 4\epsilon \delta nV'.
\end{multline*}
Dividing by $n$ gives
\[
\frac 1n \log \Lambda_n(\Lambda,\ph)
\geq \frac 1n S_n\ph(x) + \epsilon\delta(-\log \delta - 4\ell\log p - 4V') - 4\frac{\log(\epsilon n)}{n},
\]
which yields the desired result when $\delta$ is chosen sufficiently small and $n$ is chosen sufficiently large.
\end{proof}


\end{document}